\newif\ifTwoColumn

\TwoColumntrue

\ifTwoColumn
	\documentclass[10pt,twocolumn,twoside] {IEEEtran}
\else
	\documentclass[11pt]{article}
	\usepackage[hmargin = 1.0in, vmargin = 1.0in]{geometry}
\fi

\usepackage{graphicx}
\usepackage{amssymb}
\usepackage{mathtools}
\usepackage{amsmath}
\usepackage{amsthm}
\usepackage{esdiff}
\usepackage{cite}

\usepackage[us, 12hr]{datetime}

\usepackage{booktabs}

\allowdisplaybreaks		

\DeclareMathOperator{\tvd}{tvd}
\DeclareMathOperator{\sign}{sign}

\DeclarePairedDelimiter{\norm}{\lVert}{\rVert}		

\newcommand{\inv}{^{-1}}

\newcommand{\lam}{{\lambda} }

\newcommand{\iter}[1]{ ^{(#1)} }

\newcommand{\tr}{^{\mathsf{T}}}			

\newcommand{\RR}{\mathbb{R}}

\renewcommand{\le}{\leqslant}
\renewcommand{\ge}{\geqslant}

\newcommand{\thalf}{\tfrac{1}{2}}
\newcommand{\prox}{\operatorname{prox}}
\newcommand{\opt}{^{\mathsf{opt}}}			
\DeclareMathOperator{\mtvd}{mtvd}

\newtheorem{defn}{Definition}
\newtheorem{prop}{Proposition}

\newtheorem{theorem}{Theorem}

\title{Total Variation Denoising via the Moreau Envelope}

\author{Ivan Selesnick%
\thanks{The author is with the Tandon School of Engineering, New York University, New York, USA. Email: {selesi@nyu.edu}.}
\thanks{This work was supported by NSF under grant 1525398 and ONR under grant N00014-15-1-2314.}%
}

\date{Last edit: \currenttime, \today}

\begin{document}
\maketitle

\begin{abstract}

Total variation denoising is a nonlinear filtering method well suited for 
the estimation of piecewise-constant signals observed in additive white Gaussian noise. 
The method is defined by the minimization of a particular non-differentiable convex cost function.
This paper describes a generalization of this cost function
that can yield more accurate estimation of piecewise constant signals.
The new cost function involves a non-convex penalty (regularizer) 
designed to maintain the convexity of the cost function.
The new penalty is based on the Moreau envelope.
The proposed total variation denoising method
can be implemented using forward-backward splitting.

\end{abstract}

\section{Introduction}

Total variation (TV) denoising is a nonlinear filtering method based on
the assumption that the underlying signal is piecewise constant
(equivalently, the derivative of the underlying signal is \emph{sparse}) \cite{ROF_1992}.
Such signals arise in geoscience, biophysics, and other areas \cite{Little_2011_RSoc_Part1}.
The TV denoising technique is also used in conjunction with
other methods in order to process more general types of signals \cite{Gholami_2013_SP, Easley_2009_shearlet_tv, DurandFroment_2003_SIAM, Ding_2015_SPL}. 

Total variation denoising is prototypical of methods based on sparse signal models.
It is defined by the minimization of a convex cost function 
comprising a quadratic data fidelity term and a non-differentiable convex penalty term.
The penalty term 
is the composition of a linear operator and the $ \ell_1 $ norm.
Although the $ \ell_1 $ norm stands out as the convex penalty that most 
effectively induces sparsity \cite{Hastie_2015_CRC_book},
non-convex penalties 
can lead to more accurate estimation of the underlying signal \cite{Nikolova_2000_SIAM, Nikolova_2005_MMS, Nikolova_2010_TIP, Rodriguez_2009_TIP, Storath_2014_TSP}.


A few recent papers consider the prescription of non-convex penalties 
that maintain the convexity of the TV denoising cost function \cite{Lanza_2016_JMIV, MalekMohammadi_2016_TSP, Selesnick_SPL_2015, Astrom_2015_cnf}.
(The motivation for this is to leverage the benefits of both non-convex penalization and convex optimization,
e.g., to accurately estimate the amplitude of jump discontinuities while guaranteeing the uniqueness of the solution.)
The penalties considered in these works are separable (additive).
But non-separable penalties can outperform separable penalties in this context.
This is because preserving the convexity of the cost function is a severely limiting requirement. 
Non-separable penalties can more successfully meet this requirement 
because they are more general than separable penalties \cite{Selesnick_2016_TSP_BISR}.

This paper proposes a non-separable non-convex penalty for total variation denoising that generalizes the standard penalty
and maintains the convexity of the cost function to be minimized.%
\footnote{Software is available at {http://eeweb.poly.edu/iselesni/mtvd}} 
The new penalty, which is based on the Moreau envelope,
can more accurately estimate the amplitudes of jump discontinuities
in an underlying piecewise constant signal. 

\subsection{Relation to Prior Work} 

Numerous non-convex penalties and algorithms have been proposed
to outperform $ \ell_1 $-norm regularization for the estimation
of sparse signals
e.g., \cite{Castella_2015_camsap_noabr, Candes_2008_JFAP, Nikolova_2011_chap, Mohimani_2009_TSP, Marnissi_2013_ICIP, Chartrand_2014_ICASSP, Chouzenoux_2013_SIAM, Portilla_2007_SPIE, Zou_2008_AS, Chen_2014_TSP_Convergence, Wipf_2011_tinfo}.
However, few of these methods maintain the convexity of the cost function.
The prescription of non-convex penalties maintaining cost function convexity
was pioneered by Blake, Zisserman, and Nikolova \cite{Blake_1987, Nikolova_1998_ICIP, Nikolova_2010_TIP, Nikolova_2011_chap}, 
and further developed in Refs.~\cite{Bayram_2015_SPL, Bayram_2016_TSP, Chen_2014_TSP_ncogs, Ding_2015_SPL, He_2016_MSSP, Lanza_2016_JMIV, MalekMohammadi_2016_TSP, Parekh_2016_SPL_ELMA, Selesnick_2014_TSP_MSC, Selesnick_SPL_2015}.
These works rely on the presence of both strongly and weakly convex terms,
which is also exploited in \cite{Mollenhoff_2015_SIAM}.

The proposed penalty is expressed as
a differentiable convex function subtracted from the standard penalty (i.e., $ \ell_1 $ norm).
Previous works also use this idea
\cite{Parekh_2015_SPL, Selesnick_2016_TSP_BISR, Parekh_2016_SPL_ELMA}.
But the differentiable convex functions used therein are either separable
\cite{Parekh_2015_SPL, Parekh_2016_SPL_ELMA}
or sums of bivariate functions \cite{Selesnick_2016_TSP_BISR}. 

In parallel with the submission of this paper, Carlsson has also proposed using Moreau envelopes
to prescribe non-trivial convex cost functions \cite{Carlsson_2016_arxiv}.
While the approach in \cite{Carlsson_2016_arxiv} 
starts with a given non-convex cost function (e.g., with the $ \ell_0 $ pseudo-norm penalty)
and seeks the convex envelope, 
our approach starts with the $ \ell_1 $-norm penalty
and seeks a class of convexity-preserving penalties.

Some forms of generalized TV are based on infimal convolution
(related to the Moreau envelope)
\cite{Setzer_2011_CMS, Chambolle_1997_NumerMath, Burger_2016, Becker_2014_JNCA_nomonth}.
But these works
propose convex penalties suitable for non-piecewise-constant signals,
while
we propose non-convex penalties suitable for piecewise-constant signals.

\section{Total Variation Denoising}

\begin{defn}
Given $ y \in \RR^N $ and $ \lam > 0 $, 
total variation denoising is defined as
\begin{align}
	\label{eq:tvd}
	\tvd(y ; \lam)
	& =
	\arg\min_{ x \in \RR^N } 
	\bigl\{
		\thalf \norm{ y - x }_2^2 + \lam  \norm{ D x }_1
	\bigr\}
	\\
	\label{eq:prox}
	& = 
	\prox_{\lam \norm{ D \, \cdot \, }_1 }(y)
\end{align}
where
$ D $ is the $ (N-1) \times N $ matrix
\begin{equation}
	\label{eq:defD}
	D = 
	\begin{bmatrix}
		-1 & 1 & & & \\
		 & -1 & 1 & & \\
		 &  & \ddots & \ddots &  \\
		& & & -1 & 1
	\end{bmatrix}.
\end{equation}
\end{defn}

As indicated in \eqref{eq:prox},  TV denoising is the proximity operator \cite{Combettes_2011_chap}
of the function $ x \mapsto \lam \norm{ D x }_1 $.
It is convenient that TV denoising can be calculated exactly in finite-time \cite{Condat_2013, Dumbgen_2009, Johnson_2013_JCGS, Darbon_2006_JMIV_part1}. 

\section{Moreau Envelope}

Before we define the non-differentiable non-convex penalty
in Sec.~\ref{sec:pen},
we first define a differentiable convex function.
We use the Moreau envelope from convex analysis \cite{Bauschke_2011}.

\begin{defn}
Let $ \alpha \ge 0 $.
We define $ S_\alpha \colon \RR^N \to \RR $
as
\begin{equation}
	\label{eq:defS}
	S_\alpha( x ) = \min_{ v \in \RR^N }
	\bigl\{
		\norm{ D v }_1 + \tfrac{ \alpha }{ 2 } \norm{ x - v }_2^2  \, 
	\bigr\}
\end{equation}
where $ D $ is the first-order difference matrix \eqref{eq:defD}.
\end{defn}

If $ \alpha > 0 $, then $ S_\alpha $ is the \emph{Moreau envelope} of index $ \alpha\inv $ of the function $ x \mapsto \norm{ D x }_1 $.


\begin{prop}
\label{prop:calcS}
The function $ S_\alpha $ can be calculated by
\ifTwoColumn
\begin{align}
	\label{eq:zeroS}
	S_0( x )
	& = 0
	\\
	\nonumber
	S_\alpha(x)
	& = \norm{ D \tvd(x ; 1/\alpha)  }_1
	\\
	&
	\hspace{3em} {} + \tfrac{ \alpha }{ 2 } \norm{ x - \tvd(x ;  1/\alpha) }_2^2  ,
	\quad
	\alpha > 0.
\end{align}
\else
\begin{align}
	\label{eq:zeroS}
	S_0( x )
	& = 0
	\\
	S_\alpha(x)
	& = \norm{ D \tvd(x ; 1/\alpha)  }_1 + \tfrac{ \alpha }{ 2 } \norm{ x - \tvd(x ;  1/\alpha) }_2^2  ,
	\quad
	\alpha > 0.
\end{align}
\fi
\end{prop}
\begin{proof}
For $ \alpha = 0 $:
Setting $ \alpha = 0 $ and  $ v = 0 $ in \eqref{eq:defS} gives \eqref{eq:zeroS}.
For $ \alpha > 0 $:
By the definition of TV denoising, 
the $ v \in \RR^N $ minimizing the function in \eqref{eq:defS}
is the TV denoising of $ x $, i.e.,
$ v\opt = \tvd( x , 1/\alpha ) $.
\end{proof}

\begin{prop}
Let $ \alpha \ge 0 $.
The function $ S_\alpha $ satisfies
\begin{equation}
	\label{eq:boundS}
	0 \le S_\alpha(x) \le \norm{ D x }_1,
	\ \ \forall x \in \RR^N.
\end{equation}
\end{prop}
\begin{proof}
From \eqref{eq:defS}, we have
$
	S_\alpha( x )
	 \le 
	\norm{ D v }_1 + (\alpha/2) \norm{ x - v }_2^2  
$
for all $ v \in \RR^N $.
In particular, $ v = x $ leads to $ S_\alpha(x) \le \norm{ D x }_1 $.
Also, $ S_\alpha(x) \ge 0 $ since $ S_\alpha(x) $ is defined as the minimum of a non-negative function.
\end{proof}

\begin{prop}
Let $ \alpha \ge 0 $. 
The function $ S_\alpha $ is convex and differentiable. 
\end{prop}
\begin{proof}
It follows from Proposition 12.15 in Ref.~\cite{Bauschke_2011}.
\end{proof}

\begin{prop}
\label{prop:Sgrad}
Let $ \alpha \ge 0 $. 
The gradient of $ S_\alpha $ is given by
\begin{align}
	\label{eq:gradS0}
	\nabla S_0(x)
	& = 0
	\\
	\label{eq:gradS}
	\nabla S_\alpha(x)
	& =
	\alpha \bigl( x - \tvd( x ; 1 / \alpha ) \bigr),
	\quad \alpha > 0
\end{align}
where $ \tvd $ denotes total variation denoising \eqref{eq:tvd}.
\end{prop}
\begin{proof}
Since $ S_\alpha $ is the Moreau envelope of index $ \alpha\inv $ of the function $ x \mapsto \norm{ D x }_1 $ when $ \alpha > 0 $, 
it follows by Proposition 12.29 in Ref.~\cite{Bauschke_2011} that
\begin{equation}
	\nabla S_\alpha(x)
	= 
	\alpha \bigl( x - \prox_{(1/\alpha)\norm{ D \, \cdot \, }_1 }(x) \bigr).
\end{equation}
This proximity operator is TV denoising, giving \eqref{eq:gradS}.
\end{proof}

\section{Non-convex Penalty}
\label{sec:pen}

To strongly induce sparsity of $ D x $, we define a non-convex generalization of the standard TV penalty.
The new penalty is defined by subtracting a differentiable convex function from the standard penalty. 

\begin{defn}
Let $ \alpha \ge 0 $.
We define the penalty $ \psi_\alpha \colon \RR^N \to \RR $ as
\begin{equation}
	\label{eq:defpsi}
	\psi_\alpha( x ) = \norm{ D x }_1 - S_\alpha( x ) 
\end{equation}
where $ D $ is the matrix \eqref{eq:defD}
and
$ S_\alpha $ is defined by \eqref{eq:defS}.
\end{defn}

The proposed penalty is upper bounded by the standard TV penalty,
which is recovered as a special case. 

\begin{prop}
Let $ \alpha \ge 0 $. 
The penalty $ \psi_\alpha $ satisfies 
\begin{equation}
	\psi_0( x ) = \norm{ D x }_1,
	\ \ \forall x \in \RR^N
\end{equation}
and
\begin{equation}
	\label{eq:psibound}
	0 \le \psi_\alpha(x) \le \norm{ D x }_1,
	\ \ \forall x \in \RR^N.
\end{equation}
\end{prop}

\begin{proof}
It follows from \eqref{eq:zeroS} and \eqref{eq:boundS}.
\end{proof}

When a convex function is subtracted from another convex function
[as in \eqref{eq:defpsi}],
the resulting function may well be negative on part of its domain. 
Inequality \eqref{eq:psibound} states that the proposed penalty $ \psi_\alpha $ avoids this fate.
This is relevant because the penalty function should be non-negative.

Figures in the supplemental material show examples
of the proposed penalty $ \psi_\alpha $ and the function $ S_\alpha $.


\section{Enhanced TV Denoising}

We define `Moreau-enhanced' TV denoising.
If $ \alpha > 0 $, then the proposed penalty penalizes large amplitude values of $ D x $ 
less than the $ \ell_1 $ norm does (i.e., $ \psi_\alpha(x) \le \norm{D x}_1 $),
hence it is less likely to underestimate jump discontinuities.

\begin{defn}
Given $ y \in \RR^N $,  $ \lam > 0 $,  and $ \alpha \ge 0 $, 
we define Moreau-enhanced total variation denoising as
\begin{equation}
	\label{eq:mtvd}
	\mtvd( y ; \lam, \alpha) =
	\arg \min_{ x \in \RR^N }
	\bigl\{
		\thalf \norm{ y - x }_2^2 + \lam \psi_\alpha( x )
	\bigr\}
\end{equation}
where $ \psi_\alpha $ is given by \eqref{eq:defpsi}.
\end{defn}

The parameter $ \alpha $ controls the non-convexity of the penalty. 
If $ \alpha = 0 $, then the penalty is convex
and Moreau-enhanced TV denoising reduces to TV denoising.
Greater values of $ \alpha $ make the penalty more non-convex. 
What is the greatest value of $ \alpha $ that maintains convexity 
of the cost function?
The critical value is given by Theorem \ref{thm:cond}.

\begin{theorem}
\label{thm:cond}
Let $ \lam > 0 $ and $ \alpha \ge 0 $.
Define 
$ F_\alpha \colon \RR^N \to \RR $ 
as
\begin{equation}
	\label{eq:defF}
	F_\alpha(x) = \thalf \norm{ y - x }_2^2 + \lam \psi_\alpha( x )
\end{equation}
where $ \psi_\alpha $ is given by \eqref{eq:defpsi}.
If
\begin{equation}
	\label{eq:convcond}
	0 \le \alpha \le  1 / \lam 
\end{equation}
then $ F_\alpha $ is convex.
If $ 0 \le \alpha < 1/\lam $ then $ F_\alpha $ is strongly convex.
\end{theorem}
\begin{proof}
We write the cost function as
\ifTwoColumn
   \begin{align}
	F_\alpha(x) 
	& = 
	\thalf \norm{ y - x }_2^2 + \lam \norm{ D x }_1 - \lam S_\alpha( x )  
	\\
	\nonumber
	& = 
	\thalf \norm{ y - x }_2^2 + \lam \norm{ D x }_1 
		\\
		&
		\hspace{3em}
		{} - 
		 \lam \min_{ v \in \RR^N } \bigl\{  \norm{ D v }_1 + \tfrac{ \alpha }{ 2 } \norm{ x - v }_2^2  \, \bigr\}
	\\
	\nonumber
	& = 
	 \max_{ v \in \RR^N }
	 \bigl\{   
		\thalf \norm{ y - x }_2^2 + \lam \norm{ D x }_1
		\\
		&
		\hspace{6em}
		{} - 
		\lam \norm{ D v }_1 - \tfrac{ \lam \alpha }{ 2 } \norm{ x - v }_2^2  
	 \, \bigr\}
	\\
	& = 
	 \max_{ v \in \RR^N }
	 \bigl\{   
		\thalf ( 1 - \lam \alpha ) \norm{ x }_2^2 + \lam \norm{ D x }_1 + g(x, v)
	 \, \bigr\}
	\\
	& = 
	\thalf ( 1 - \lam \alpha ) \norm{ x }_2^2 + \lam \norm{ D x }_1
	+
	\max_{ v \in \RR^N }
	g(x, v)
   \end{align}
\else
   \begin{align}
	F_\alpha(x) 
	& = 
	\thalf \norm{ y - x }_2^2 + \lam \norm{ D x }_1 - \lam S_\alpha( x )  
	\\
	& = 
	\thalf \norm{ y - x }_2^2 + \lam \norm{ D x }_1 - 
		 \lam \min_{ v \in \RR^N } \bigl\{  \norm{ D v }_1 + \tfrac{ \alpha }{ 2 } \norm{ x - v }_2^2  \, \bigr\}
	\\
	& = 
	 \max_{ v \in \RR^N }
	 \bigl\{   
		\thalf \norm{ y - x }_2^2 + \lam \norm{ D x }_1 - 
		\lam \norm{ D v }_1 - \tfrac{ \lam \alpha }{ 2 } \norm{ x - v }_2^2  
	 \, \bigr\}
	\\
	& = 
	 \max_{ v \in \RR^N }
	 \bigl\{   
		\thalf ( 1 - \lam \alpha ) \norm{ x }_2^2 + \lam \norm{ D x }_1 + g(x, v)
	 \, \bigr\}
	\\
	& = 
	\thalf ( 1 - \lam \alpha ) \norm{ x }_2^2 + \lam \norm{ D x }_1
	+
	\max_{ v \in \RR^N }
	g(x, v)
   \end{align}
\fi
where $ g(x, v) $ is affine in $ x $.
The last term is convex as it is the point-wise maximum of a set of convex functions.
Hence, $ F_\alpha $ is a convex function if $ 1 - \lam \alpha \ge 0 $.
If  $ 1 - \lam \alpha > 0 $, then
$ F_\alpha $ is strongly convex (and strictly convex).
\end{proof}

\section{Algorithm}
\label{sec:alg}

\begin{prop}
Let $ y \in \RR^N $,  $ \lam > 0 $,  and $ 0 < \alpha < 1/\lam $.
Then $ x\iter{k} $ produced by the iteration
\begin{subequations}
\label{eq:alg}
\begin{align}
	z\iter{k} & = 
		y + \lam \alpha \bigl(  x\iter{k} -  \tvd( x\iter{k} ; 1 / \alpha ) \bigr)
	\\
	x\iter{k+1} & =
	\tvd( z\iter{k} ; \lam ).
\end{align}
\end{subequations}
converges to the solution of the Moreau-enhanced TV denoising problem \eqref{eq:mtvd}.
\end{prop}

\begin{proof}
If the cost function \eqref{eq:defF} is strongly convex, then the minimizer 
can be calculated using 
the forward-backward splitting (FBS) algorithm \cite{Bauschke_2011, Combettes_2011_chap}.
This algorithm minimizes a function of the form
\begin{equation}
	\label{eq:FBSd}
	F(x) = f_1(x) + f_2(x)
\end{equation}
where
both $ f_1 $ and $ f_2 $ are convex and $ \nabla f_1 $ is additionally Lipschitz continuous.
The FBS algorithm is given by
\begin{subequations}
\label{eq:fbs}
\begin{align}
	z\iter{k} & = 
	x\iter{k} - \mu \bigl[  \nabla f_1( x\iter{k} ) \bigr]
	\\
	x\iter{k+1} & = \arg \min_x \big\{ \thalf \norm{ z\iter{k} - x }_2^2 + \mu f_2( x ) \big\}
\end{align}
\end{subequations}
where
$ 0 < \mu < 2 / \rho $
and
$ \rho $ is the Lipschitz constant of $ \nabla f_1 $.
The iterates $ x\iter{k} $ 
 converge to a minimizer of $ F $.

To apply the FBS algorithm to the proposed cost function \eqref{eq:defF},
we write it as
\begin{align}
	F_\alpha(x) 
	& = 
	\thalf \norm{ y - x }_2^2 + \lam \psi_\alpha( x ),
	\\
	& = 
	\thalf \norm{ y - x }_2^2 + \lam \norm{ D x }_1 - \lam S_\alpha( x )  
	\\
	& = 
	f_1(x) + f_2(x)
\end{align}
where
\begin{subequations}
\label{eq:deff12}
\begin{align}
	\label{eq:deff1}
	f_1(x)
	& = 
	\thalf \norm{ y - x }_2^2
	-
	\lam S_\alpha( x )  
	\\
	\label{eq:deff2}
	f_2(x)
	& = 
	\lam \norm{ D x }_1.
\end{align}
\end{subequations}
The gradient of $ f_1 $ is given by
\begin{align}
	\nabla f_1(x) 
	& =
	x - y - \lam \nabla S_\alpha( x )
	\\
	& =
	x - y - \lam \alpha \bigl(  x -  \tvd( x ; 1 / \alpha ) \bigr)
\end{align}
using Proposition \ref{prop:Sgrad}.
Subtracting $ S_\alpha $ from $ f_1 $ does not increase the Lipschitz constant of $ \nabla f_1 $,
the value of which is 1.
Hence, we may set $ 0 < \mu < 2 $.

Using \eqref{eq:deff12}, the FBS algorithm \eqref{eq:fbs} becomes
\ifTwoColumn
   \begin{subequations}
   \begin{align}
   	\nonumber
	z\iter{k} & = 
	x\iter{k} - \mu \bigl[
		x\iter{k} - y
	\\
		& \hspace{4em} {} - \lam \alpha \bigl(  x\iter{k} -  \tvd( x\iter{k} ; 1 / \alpha ) \bigr)
	\bigr]
	\\
	\label{eq:xupdate}
	x\iter{k+1} & = \arg \min_x \big\{ \thalf \norm{ z\iter{k} - x }_2^2 + \mu \lam \norm{ D x }_1 \big\}.
   \end{align}
   \end{subequations}
\else
   \begin{subequations}
   \begin{align}
	z\iter{k} & = 
	x\iter{k} - \mu \bigl[
		x\iter{k} - y - \lam \alpha \bigl(  x\iter{k} -  \tvd( x\iter{k} ; 1 / \alpha ) \bigr)
	\bigr]
	\\
	\label{eq:xupdate}
	x\iter{k+1} & = \arg \min_x \big\{ \thalf \norm{ z\iter{k} - x }_2^2 + \mu \lam \norm{ D x }_1 \big\}.
   \end{align}
   \end{subequations}
\fi
Note that \eqref{eq:xupdate} is TV denoising \eqref{eq:tvd}.
Using the value $ \mu = 1 $ gives iteration \eqref{eq:alg}.
(Experimentally, we found this value yields fast convergence.) 
\end{proof}

Each iteration of \eqref{eq:alg} entails solving two standard TV denoising problems.
In this work, we calculate TV denoising using the fast exact C language program by Condat \cite{Condat_2013}.
Like the
iterative shrinkage/thresholding algorithm (ISTA) \cite{Daubechies_2004, Fig_2003_TIP},
algorithm \eqref{eq:alg} 
can be accelerated in various ways.

 We suggest not setting $ \alpha $ too close to the critical value $ 1/\lam $ 
because the FBS algorithm generally converges faster when 
the cost function is more strongly convex ($ \alpha < 1 $).

In summary, the proposed
Moreau-enhanced TV denoising method comprises the steps:
\begin{enumerate}
\item
Set the regularization parameter $ \lam $ ($ \lam > 0 $).
\item
Set the non-convexity parameter $ \alpha $ ($ 0 \le \alpha < 1/\lam $).
\item
Initialize $ x\iter{0} = 0 $.
\item
Run iteration \eqref{eq:alg} until convergence.
\end{enumerate}

\section{Optimality Condition}

To avoid terminating the iterative algorithm too early, 
it is useful to verify convergence using an optimality condition.

\begin{prop}
\label{prop:opt}
Let $ y \in \RR^N $,  $ \lam > 0 $,  and $ 0 < \alpha < 1/\lam $.
If $ x $ is a solution to \eqref{eq:mtvd},
then
\begin{equation}
	\label{eq:opt}
	\bigl[ C \bigl( (x - y)/\lam + \alpha (  \tvd( x ; 1 / \alpha ) - x )  \bigr)     \bigr]_n 
	\in
	\sign( [ D x ]_n )
\end{equation}
for $ n = 0, \dots, N-1 $,
where $ C \in \RR^{ (N-1) \times N } $ is
given by
\begin{equation}
	\label{eq:defC}
	C_{m, n} =
	\begin{cases}
		1, \  & m \ge n
		\\
		0, & m < n,
	\end{cases}
	\quad
	\text{i.e.,}
	\quad	
	[C x]_n = \sum_{ m \le n } x_m
\end{equation}
and $ \sign $ is the set-valued signum function
\begin{equation}
	\sign(t) = 
	\begin{cases}
		\{ -1 \}, \ \ & t < 0
		\\
		[-1, 1], & t = 0
		\\
		\{ 1 \}, & t > 0.
	\end{cases}
\end{equation}
\end{prop}

According to \eqref{eq:opt},
if $ x \in \RR^N $ is a minimizer, then
the points $ ( [Dx]_n, u_n) \in \RR^2 $ must lie on the graph of the signum function,
where $ u_n $ denotes the value on the left-hand side of \eqref{eq:opt}.
Hence, the optimality condition can be depicted as a scatter plot. 
Figures in the supplemental material show how the points in the scatter plot converge
to the signum function as the algorithm \eqref{eq:alg} progresses.

\begin{proof}[Proof of Proposition \ref{prop:opt}]
A vector $ x $ minimizes a convex function $ F $ if $ 0 \in \partial F( x ) $
where $ \partial F(x) $ is the subdifferential of $ F $ at $ x $.
The subdifferential of the cost function \eqref{eq:defF} is given by
\begin{equation}
	\partial F_\alpha( x ) = x - y - \lam \nabla S_\alpha(x) + \partial (  \lam \norm{ D \, \cdot \, }_1 )(x)
\end{equation}
which can be written as
\ifTwoColumn
	\begin{multline}
		\partial F_\alpha( x ) = \{ x - y - \lam \nabla S_\alpha(x) +  \lam D\tr \!  u
		\\
		{}  :  u_n \in \sign( [ D x ]_n ) , \, u \in \RR^{N-1} \}.
	\end{multline}
\else
	\begin{equation}
		\partial F_\alpha( x ) = \{ x - y - \lam \nabla S_\alpha(x) +  \lam D\tr \!  u  :  u_n \in \sign( [ D x ]_n ) , \, u \in \RR^{N-1} \}.
	\end{equation}
\fi
Hence, the condition $ 0 \in \partial F_\alpha( x ) $ can be written as
\ifTwoColumn
	\begin{multline}
		(y - x)/\lam + \nabla S_\alpha(x)
		\\
		{} \in
		\{ D\tr \! u :  u_n \in \sign( [ D x]_n ) , \, u \in \RR^{N-1}  \}.
	\end{multline}
\else
	\begin{equation}
		(y - x)/\lam + \nabla S_\alpha(x)
		\in
		\{ D\tr \! u :  u_n \in \sign( [ D x]_n ) , \, u \in \RR^{N-1}  \}.
	\end{equation}
\fi

Let $ C $ be a matrix of size $ (N-1) \times N $ such that  $ C D\tr = -I $, 
e.g., \eqref{eq:defC}.
It follows that the condition $ 0 \in \partial F_\alpha( x ) $ implies that
\begin{equation}
	 \bigl[ C \bigl( (x - y)/\lam - \nabla S_\alpha(x) \bigr) \bigr]_n 
	\in
	\sign( [ D x]_n )
\end{equation}
for $ n = 0, \dots, N-1 $.
Using Proposition \ref{prop:Sgrad} gives \eqref{eq:opt}.
\end{proof}

\section{Example}

\begin{figure}[t]
	\centering
	\includegraphics{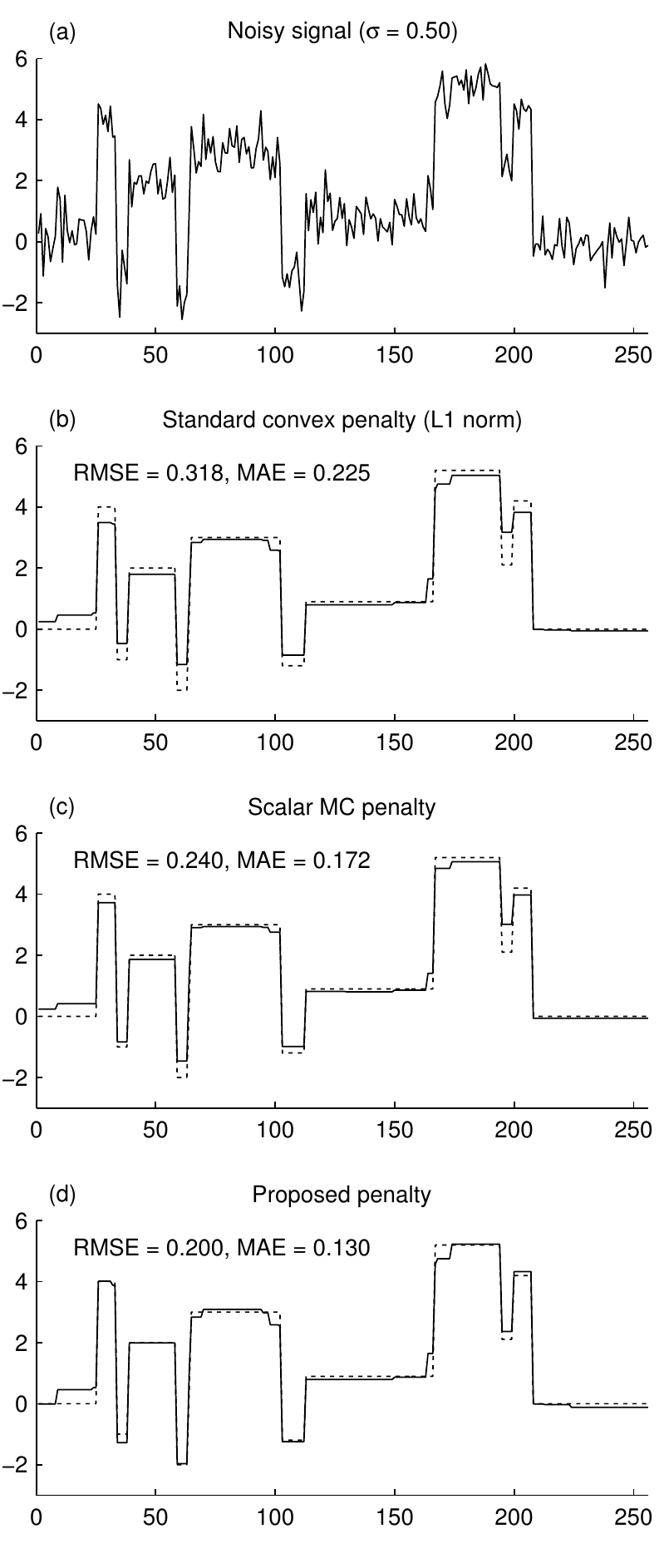}
	\caption{
		Total variation denoising using three different penalties. 
		(The dashed line is the true noise-free signal.)
	}
	\label{fig:example1}
\end{figure}

This example applies TV denoising to the noisy piecewise constant signal 
shown in Fig.~\ref{fig:example1}(a).
This is the `blocks'  signal (length $ N = 256 $) 
generated by the Wavelab  \cite{wavelab} function \texttt{MakeSignal}
with additive white Gaussian noise ($ \sigma = 0.5 $).
We set the regularization parameter to $ \lam = \sqrt{N} \sigma /4 $
following a discussion in Ref.~\cite{Dumbgen_2009}.
For Moreau-enhanced TV denoising,
we set the non-convexity parameter to $ \alpha = 0.7 / \lam $.

Figure~\ref{fig:example1} shows the result of TV denoising with three different penalties.
In each case, a \emph{convex} cost function is minimized. 
Figure~\ref{fig:example1}(b)
shows the result using standard TV denoising (i.e., using the $ \ell_1 $-norm). 
This denoised signal consistently underestimates the amplitudes 
of jump discontinuities, especially those occurring near other jump discontinuities of opposite sign. 
Figure~\ref{fig:example1}(c)
shows the result using a separable non-convex penalty \cite{Selesnick_SPL_2015}.
This method can use any non-convex scalar penalty satisfying a prescribed set of properties. 
Here we use the minimax-concave (MC) penalty~\cite{Zhang_2010_AnnalsStat, Bayram_2015_SPL}
with non-convexity parameter set to maintain cost function convexity. 
This result significantly improves the root-mean-square error (RMSE)
and mean-absolute-deviation (MAE), 
but still underestimates the amplitudes of jump discontinuities.

Moreau-enhanced TV denoising, shown in Fig.~\ref{fig:example1}(d),
further reduces the RMSE and MAE and more accurately estimates
the amplitudes of jump discontinuities. 
The proposed non-separable non-convex penalty 
avoids the consistent underestimation of discontinuities
seen in Figs.~\ref{fig:example1}(b) and \ref{fig:example1}(c).

To further compare the denoising capability of the considered penalties, 
we calculate the average RMSE as a function of the noise level. 
We let the noise standard deviation span the interval $ 0.2 \le \sigma \le 1.0 $.
For each $ \sigma $ value, we calculate the average RMSE of 100 noise realizations.
Figure~\ref{fig:rmse} shows that the proposed penalty yields
the lowest average RMSE for all $ \sigma \ge 0.4 $. 
However, at low noise levels, separable convexity-preserving penalties \cite{Selesnick_SPL_2015}
perform better than the proposed non-separable convexity-preserving penalty.

\begin{figure}[t]
	\centering
	\includegraphics{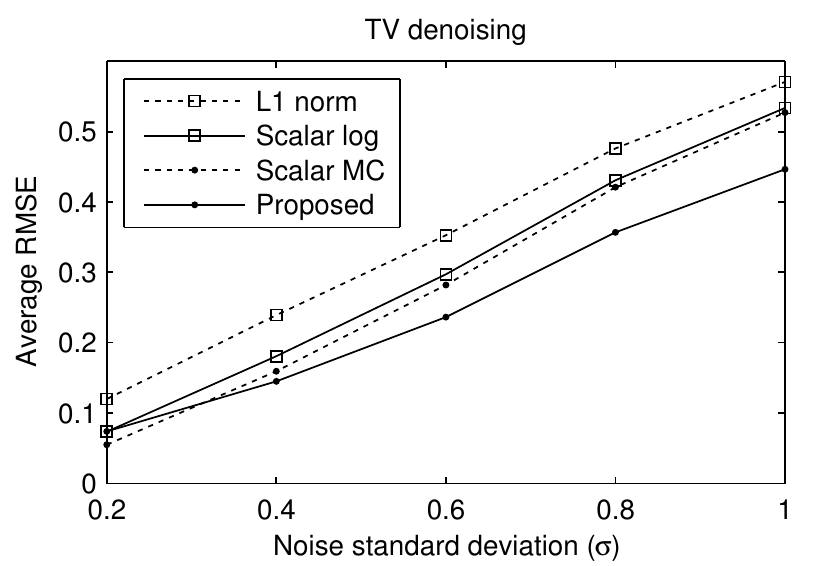}
	\caption{
		TV denoising using four penalties:
		RMSE as a function of noise level.
	}
	\label{fig:rmse}
\end{figure}

\section{Conclusion}

This paper demonstrates the use of the Moreau envelope to 
define a non-separable non-convex TV denoising penalty 
that maintains the convexity of the TV denoising cost function. 
The basic idea is to subtract from a convex penalty its Moreau envelope.
This idea should also be useful for other problems, 
e.g., analysis tight-frame denoising \cite{Parekh_2015_SPL}.

Separable convexity-preserving penalties \cite{Selesnick_SPL_2015}
outperformed the proposed one at low noise levels in the example.
It is yet to be determined if a more general class of convexity-preserving penalties can
outperform both across all noise levels. 

\bibliographystyle{plain}

\begin{thebibliography}{10}

\bibitem{Astrom_2015_cnf}
F.~Astrom and C.~Schnorr.
\newblock On coupled regularization for non-convex variational image
  enhancement.
\newblock In {\em IAPR Asian Conf. on Pattern Recognition (ACPR)}, pages
  786--790, November 2015.

\bibitem{Bauschke_2011}
H.~H. Bauschke and P.~L. Combettes.
\newblock {\em Convex Analysis and Monotone Operator Theory in Hilbert Spaces}.
\newblock Springer, 2011.

\bibitem{Bayram_2015_SPL}
\.{I}. Bayram.
\newblock Penalty functions derived from monotone mappings.
\newblock {\em IEEE Signal Processing Letters}, 22(3):265--269, March 2015.

\bibitem{Bayram_2016_TSP}
I.~Bayram.
\newblock On the convergence of the iterative shrinkage/thresholding algorithm
  with a weakly convex penalty.
\newblock {\em IEEE Trans.\ Signal Process.}, 64(6):1597--1608, March 2016.

\bibitem{Becker_2014_JNCA_nomonth}
S.~Becker and P.~L. Combettes.
\newblock An algorithm for splitting parallel sums of linearly composed
  monotone operators, with applications to signal recovery.
\newblock {\em J. Nonlinear and Convex Analysis}, 15(1):137--159, 2014.

\bibitem{Blake_1987}
A.~Blake and A.~Zisserman.
\newblock {\em Visual Reconstruction}.
\newblock MIT Press, 1987.

\bibitem{Burger_2016}
M.~Burger, K.~Papafitsoros, E.~Papoutsellis, and C.-B. Sch{\"o}nlieb.
\newblock Infimal convolution regularisation functionals of {BV} and {L}p
  spaces.
\newblock {\em J. Math. Imaging and Vision}, 55(3):343--369, 2016.

\bibitem{Candes_2008_JFAP}
E.~J. Cand\`es, M.~B. Wakin, and S.~Boyd.
\newblock Enhancing sparsity by reweighted l1 minimization.
\newblock {\em J. {F}ourier Anal. Appl.}, 14(5):877--905, December 2008.

\bibitem{Carlsson_2016_arxiv}
M.~Carlsson.
\newblock On convexification/optimization of functionals including an l2-misfit
  term.
\newblock https://arxiv.org/abs/1609.09378, September 2016.

\bibitem{Castella_2015_camsap_noabr}
M.~Castella and J.-C. Pesquet.
\newblock Optimization of a {G}eman-{M}c{C}lure like criterion for sparse
  signal deconvolution.
\newblock In {\em IEEE Int. Workshop Comput. Adv. Multi-Sensor Adaptive Proc.},
  pages 309--312, December 2015.

\bibitem{Chambolle_1997_NumerMath}
A.~Chambolle and P.-L. Lions.
\newblock Image recovery via total variation minimization and related problems.
\newblock {\em Numerische Mathematik}, 76:167--188, 1997.

\bibitem{Chartrand_2014_ICASSP}
R.~Chartrand.
\newblock Shrinkage mappings and their induced penalty functions.
\newblock In {\em Proc.~IEEE Int.~Conf. Acoust., Speech, Signal Processing
  (ICASSP)}, pages 1026--1029, May 2014.

\bibitem{Chen_2014_TSP_Convergence}
L.~Chen and Y.~Gu.
\newblock The convergence guarantees of a non-convex approach for sparse
  recovery.
\newblock {\em IEEE Trans.\ Signal Process.}, 62(15):3754--3767, August 2014.

\bibitem{Chen_2014_TSP_ncogs}
P.-Y. Chen and I.~W. Selesnick.
\newblock Group-sparse signal denoising: Non-convex regularization, convex
  optimization.
\newblock {\em IEEE Trans.\ Signal Process.}, 62(13):3464--3478, July 2014.

\bibitem{Chouzenoux_2013_SIAM}
E.~Chouzenoux, A.~Jezierska, J.~Pesquet, and H.~Talbot.
\newblock A majorize-minimize subspace approach for $\ell_2-\ell_0$ image
  regularization.
\newblock {\em SIAM J. Imag. Sci.}, 6(1):563--591, 2013.

\bibitem{Combettes_2011_chap}
P.~L. Combettes and J.-C. Pesquet.
\newblock Proximal splitting methods in signal processing.
\newblock In H.~H. Bauschke et~al., editors, {\em Fixed-Point Algorithms for
  Inverse Problems in Science and Engineering}, pages 185--212.
  Springer-Verlag, 2011.

\bibitem{Condat_2013}
L.~Condat.
\newblock A direct algorithm for {1-D} total variation denoising.
\newblock {\em IEEE Signal Processing Letters}, 20(11):1054--1057, November
  2013.

\bibitem{Darbon_2006_JMIV_part1}
J.~Darbon and M.~Sigelle.
\newblock Image restoration with discrete constrained total variation {P}art
  {I}: {F}ast and exact optimization.
\newblock {\em J. Math. Imaging and Vision}, 26(3):261--276, 2006.

\bibitem{Daubechies_2004}
I.~Daubechies, M.~Defrise, and C.~De Mol.
\newblock An iterative thresholding algorithm for linear inverse problems with
  a sparsity constraint.
\newblock {\em Commun. Pure Appl. Math}, 57(11):1413--1457, 2004.

\bibitem{Ding_2015_SPL}
Y.~Ding and I.~W. Selesnick.
\newblock Artifact-free wavelet denoising: Non-convex sparse regularization,
  convex optimization.
\newblock {\em IEEE Signal Processing Letters}, 22(9):1364--1368, September
  2015.

\bibitem{wavelab}
D.~Donoho, A.~Maleki, and M.~Shahram.
\newblock Wavelab 850, 2005.
\newblock \texttt{http://www-stat.stanford.edu/\%7Ewavelab/}.

\bibitem{Dumbgen_2009}
L.~D\"umbgen and A.~Kovac.
\newblock Extensions of smoothing via taut strings.
\newblock {\em Electron. J. Statist.}, 3:41--75, 2009.

\bibitem{DurandFroment_2003_SIAM}
S.~Durand and J.~Froment.
\newblock Reconstruction of wavelet coefficients using total variation
  minimization.
\newblock {\em SIAM J. Sci. Comput.}, 24(5):1754--1767, 2003.

\bibitem{Easley_2009_shearlet_tv}
G.~R. Easley, D.~Labate, and F.~Colonna.
\newblock Shearlet-based total variation diffusion for denoising.
\newblock {\em IEEE Trans.\ Image Process.}, 18(2):260--268, February 2009.

\bibitem{Fig_2003_TIP}
M.~Figueiredo and R.~Nowak.
\newblock An {EM} algorithm for wavelet-based image restoration.
\newblock {\em IEEE Trans.\ Image Process.}, 12(8):906--916, August 2003.

\bibitem{Gholami_2013_SP}
A.~Gholami and S.~M. Hosseini.
\newblock A balanced combination of {T}ikhonov and total variation
  regularizations for reconstruction of piecewise-smooth signals.
\newblock {\em Signal Processing}, 93(7):1945--1960, 2013.

\bibitem{Hastie_2015_CRC_book}
T.~Hastie, R.~Tibshirani, and M.~Wainwright.
\newblock {\em Statistical learning with sparsity: the lasso and
  generalizations}.
\newblock CRC Press, 2015.

\bibitem{He_2016_MSSP}
W.~He, Y.~Ding, Y.~Zi, and I.~W. Selesnick.
\newblock Sparsity-based algorithm for detecting faults in rotating machines.
\newblock {\em Mechanical Systems and Signal Processing}, 72-73:46--64, May
  2016.

\bibitem{Johnson_2013_JCGS}
N.~A. Johnson.
\newblock A dynamic programming algorithm for the fused lasso and
  ${L}_0$-segmentation.
\newblock {\em J. Computat. Graph. Stat.}, 22(2):246--260, 2013.

\bibitem{Lanza_2016_JMIV}
A.~Lanza, S.~Morigi, and F.~Sgallari.
\newblock Convex image denoising via non-convex regularization with parameter
  selection.
\newblock {\em J. Math. Imaging and Vision}, pages 1--26, 2016.

\bibitem{Little_2011_RSoc_Part1}
M.~A. Little and N.~S. Jones.
\newblock Generalized methods and solvers for noise removal from piecewise
  constant signals: Part {I} -- background theory.
\newblock {\em Proc. R. Soc. A}, 467:3088--3114, 2011.

\bibitem{MalekMohammadi_2016_TSP}
M.~Malek-Mohammadi, C.~R. Rojas, and B.~Wahlberg.
\newblock A class of nonconvex penalties preserving overall convexity in
  optimization-based mean filtering.
\newblock {\em IEEE Trans.\ Signal Process.}, 64(24):6650--6664, December 2016.

\bibitem{Marnissi_2013_ICIP}
Y.~Marnissi, A.~Benazza-Benyahia, E.~Chouzenoux, and J.-C. Pesquet.
\newblock Generalized multivariate exponential power prior for wavelet-based
  multichannel image restoration.
\newblock In {\em Proc.~IEEE Int.~Conf.~Image Processing (ICIP)}, pages
  2402--2406, September 2013.

\bibitem{Mohimani_2009_TSP}
H.~Mohimani, M.~Babaie-Zadeh, and C.~Jutten.
\newblock A fast approach for overcomplete sparse decomposition based on
  smoothed l0 norm.
\newblock {\em IEEE Trans.\ Signal Process.}, 57(1):289--301, January 2009.

\bibitem{Mollenhoff_2015_SIAM}
T.~M\"ollenhoff, E.~Strekalovskiy, M.~Moeller, and D.~Cremers.
\newblock The primal-dual hybrid gradient method for semiconvex splittings.
\newblock {\em SIAM J. Imag. Sci.}, 8(2):827--857, 2015.

\bibitem{Nikolova_1998_ICIP}
M.~Nikolova.
\newblock Estimation of binary images by minimizing convex criteria.
\newblock In {\em Proc.~IEEE Int.~Conf.~Image Processing (ICIP)}, pages
  108--112 vol. 2, 1998.

\bibitem{Nikolova_2000_SIAM}
M.~Nikolova.
\newblock Local strong homogeneity of a regularized estimator.
\newblock {\em SIAM J. Appl. Math.}, 61(2):633--658, 2000.

\bibitem{Nikolova_2005_MMS}
M.~Nikolova.
\newblock Analysis of the recovery of edges in images and signals by minimizing
  nonconvex regularized least-squares.
\newblock {\em Multiscale Model. Simul.}, 4(3):960--991, 2005.

\bibitem{Nikolova_2011_chap}
M.~Nikolova.
\newblock Energy minimization methods.
\newblock In O.~Scherzer, editor, {\em Handbook of Mathematical Methods in
  Imaging}, chapter~5, pages 138--186. Springer, 2011.

\bibitem{Nikolova_2010_TIP}
M.~Nikolova, M.~K. Ng, and C.-P. Tam.
\newblock Fast nonconvex nonsmooth minimization methods for image restoration
  and reconstruction.
\newblock {\em IEEE Trans.\ Image Process.}, 19(12):3073--3088, December 2010.

\bibitem{Parekh_2015_SPL}
A.~Parekh and I.~W. Selesnick.
\newblock Convex denoising using non-convex tight frame regularization.
\newblock {\em IEEE Signal Processing Letters}, 22(10):1786--1790, October
  2015.

\bibitem{Parekh_2016_SPL_ELMA}
A.~Parekh and I.~W. Selesnick.
\newblock Enhanced low-rank matrix approximation.
\newblock {\em IEEE Signal Processing Letters}, 23(4):493--497, April 2016.

\bibitem{Portilla_2007_SPIE}
J.~Portilla and L.~Mancera.
\newblock {L0}-based sparse approximation: two alternative methods and some
  applications.
\newblock In {\em Proceedings of SPIE}, volume 6701 (Wavelets XII), San Diego,
  CA, USA, 2007.

\bibitem{Rodriguez_2009_TIP}
P.~Rodriguez and B.~Wohlberg.
\newblock Efficient minimization method for a generalized total variation
  functional.
\newblock {\em IEEE Trans.\ Image Process.}, 18(2):322--332, February 2009.

\bibitem{ROF_1992}
L.~Rudin, S.~Osher, and E.~Fatemi.
\newblock Nonlinear total variation based noise removal algorithms.
\newblock {\em Physica D}, 60:259--268, 1992.

\bibitem{Selesnick_2014_TSP_MSC}
I.~W. Selesnick and I.~Bayram.
\newblock Sparse signal estimation by maximally sparse convex optimization.
\newblock {\em IEEE Trans.\ Signal Process.}, 62(5):1078--1092, March 2014.

\bibitem{Selesnick_2016_TSP_BISR}
I.~W. Selesnick and I.~Bayram.
\newblock Enhanced sparsity by non-separable regularization.
\newblock {\em IEEE Trans.\ Signal Process.}, 64(9):2298--2313, May 2016.

\bibitem{Selesnick_SPL_2015}
I.~W. Selesnick, A.~Parekh, and I.~Bayram.
\newblock Convex {1-D} total variation denoising with non-convex
  regularization.
\newblock {\em IEEE Signal Processing Letters}, 22(2):141--144, February 2015.

\bibitem{Setzer_2011_CMS}
S.~Setzer, G.~Steidl, and T.~Teuber.
\newblock Infimal convolution regularizations with discrete l1-type
  functionals.
\newblock {\em Commun. Math. Sci.}, 9(3):797--827, 2011.

\bibitem{Storath_2014_TSP}
M.~Storath, A.~Weinmann, and L.~Demaret.
\newblock Jump-sparse and sparse recovery using {P}otts functionals.
\newblock {\em IEEE Trans.\ Signal Process.}, 62(14):3654--3666, July 2014.

\bibitem{Wipf_2011_tinfo}
D.~P. Wipf, B.~D. Rao, and S.~Nagarajan.
\newblock Latent variable {B}ayesian models for promoting sparsity.
\newblock {\em IEEE Trans.~Inform. Theory}, 57(9):6236--6255, September 2011.

\bibitem{Zhang_2010_AnnalsStat}
C.-H. Zhang.
\newblock Nearly unbiased variable selection under minimax concave penalty.
\newblock {\em The Annals of Statistics}, pages 894--942, 2010.

\bibitem{Zou_2008_AS}
H.~Zou and R.~Li.
\newblock One-step sparse estimates in nonconcave penalized likelihood models.
\newblock {\em Ann. Statist.}, 36(4):1509--1533, 2008.

\end{thebibliography}

\clearpage

\section{Supplemental Figures}

To gain intuition about the proposed penalty function and 
how it induces sparsity of $ D x $ while maintaining
convexity of the cost function,
a few illustrations are useful.

Figure~\ref{fig:pen1} illustrates the proposed penalty,
its sparsity-inducing behavior, and its relationship to the differentiable convex function $ S_\alpha $.
Figure~\ref{fig:pen2} illustrates how the proposed penalty 
is able to maintain the convexity of the cost function.

Figure~\ref{fig:pen1}
shows the proposed penalty $ \psi_\alpha $ defined in \eqref{eq:defpsi}
for $ \alpha = 1 $ and $ N = 2 $.
It can be seen that the penalty approximates 
the standard TV penalty $ \norm{ D \, \cdot \,}_1 $ for signals $ x $ 
for which $ D x $ is approximately zero.
But it increases more slowly than the standard TV penalty as $ \norm{ D x } \to \infty $.
In that sense, it penalizes large values less than the standard TV penalty.

As shown in Fig.~\ref{fig:pen1}, the proposed penalty is expressed as the standard TV penalty 
minus the differentiable convex non-negative function $ S_\alpha $.
Since $ S_\alpha $ is flat around the null space of $ D $,
the penalty $ \psi_\alpha $ approximates the standard TV penalty
around the null space of $ D $.
In addition, since $ S_\alpha $ is non-negative, 
the penalty $ \psi_\alpha $ lies below the standard TV penalty.

Figure~\ref{fig:pen2}
shows the differentiable part of the cost function $ F_\alpha $
for $ \alpha = 1 $, $ \lam = 1 $, and $ N = 2 $.
The differentiable part is given by $ f_1 $ in  \eqref{eq:deff1}.
The total cost function is obtained by adding the standard TV penalty to $ f_1 $, 
see \eqref{eq:FBSd}.
Hence, $ F_\alpha $ is convex if the differentiable part $ f_1 $ is convex.
As can be seen in Fig.~\ref{fig:pen2}, the function $ f_1 $ is convex. 
We note that the function $ f_1 $ in this figure is not strongly convex.
This is because we have used $ \alpha = 1/\lam $.
If $ \alpha < 1/\lam $, then the function $ f_1 $ will be strongly convex
(and hence $ F_\alpha $ will also be strongly convex and have a unique minimizer).
We recommend $ \alpha < 1/\lam $.

Figure~\ref{fig:pen3}
shows the differentiable part $ f_1 $ of the cost function $ F_\alpha $
for $ \alpha = 2 $, $ \lam = 1 $, and $ N = 2 $.
Here, the function $ f_1 $ is non-convex because $ \alpha > 1/\lam $
which violates the convexity condition. 

In order to simplify the illustration, we have set $ y = 0 $ in Fig.~\ref{fig:pen2}.
In practice $ y \neq 0 $.
But the only difference between the cases $ y = 0 $ and $ y \neq 0 $
is an additive affine function
which does not alter the convexity properties of the function. 

In practice we are interested in the case $ N \gg 2 $,
i.e., signals much longer than two samples.
However, in order to illustrate the functions,
we are limited to the case of $ N = 2 $.
We note that the case of $ N = 2 $ does not fully illustrate the
behavior of the proposed penalty.
In particular, when $ N = 2 $
the penalty is simply a linear transformation of a scalar function,
which does not convey the non-separable behavior
of the penalty for $ N > 2 $.

A separate document has additional supplemental figures 
illustrating the convergence of the iterative algorithm \eqref{eq:alg}.
These figures show the optimality condition as a scatter plot.
The points in the scatter plot converge
to the signum function as the algorithm converges.

\begin{figure}[t]
	\centering
	\includegraphics{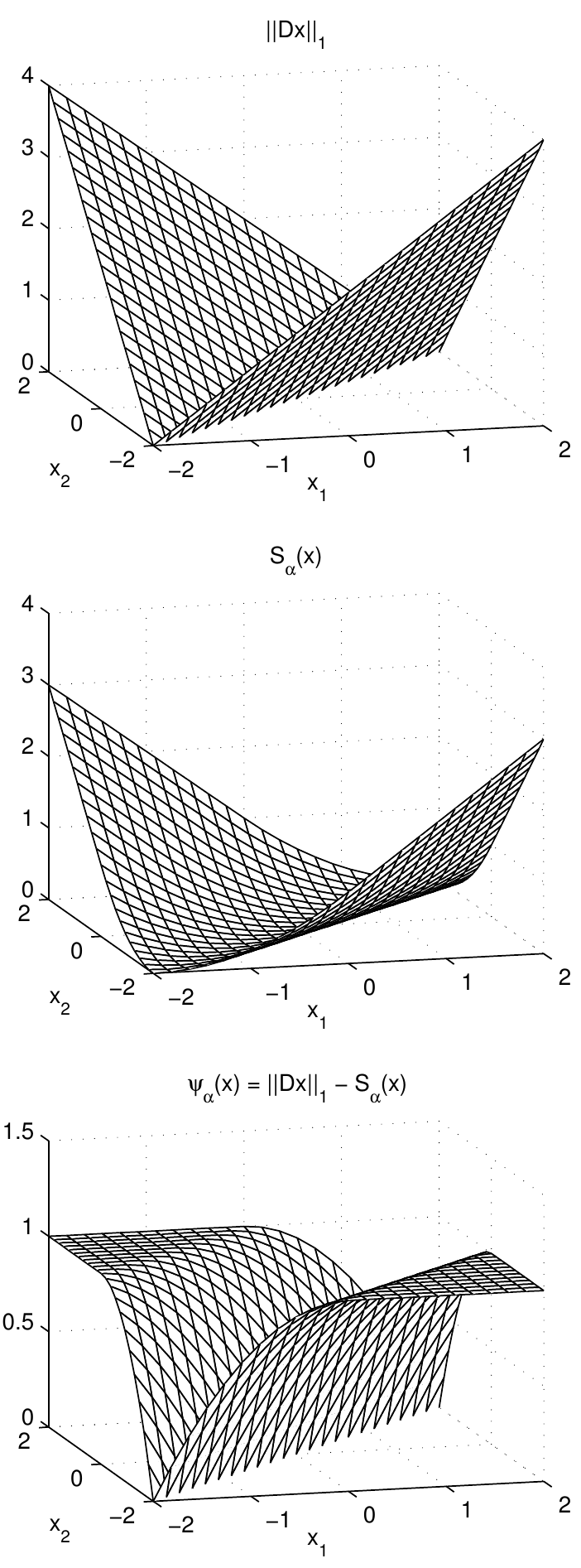}
	\caption{
		Penalty $ \psi_\alpha $ with $ \alpha = 1 $ for $ N = 2 $.
	}
	\label{fig:pen1}
\end{figure}

\begin{figure}[t]
	\centering
	\includegraphics{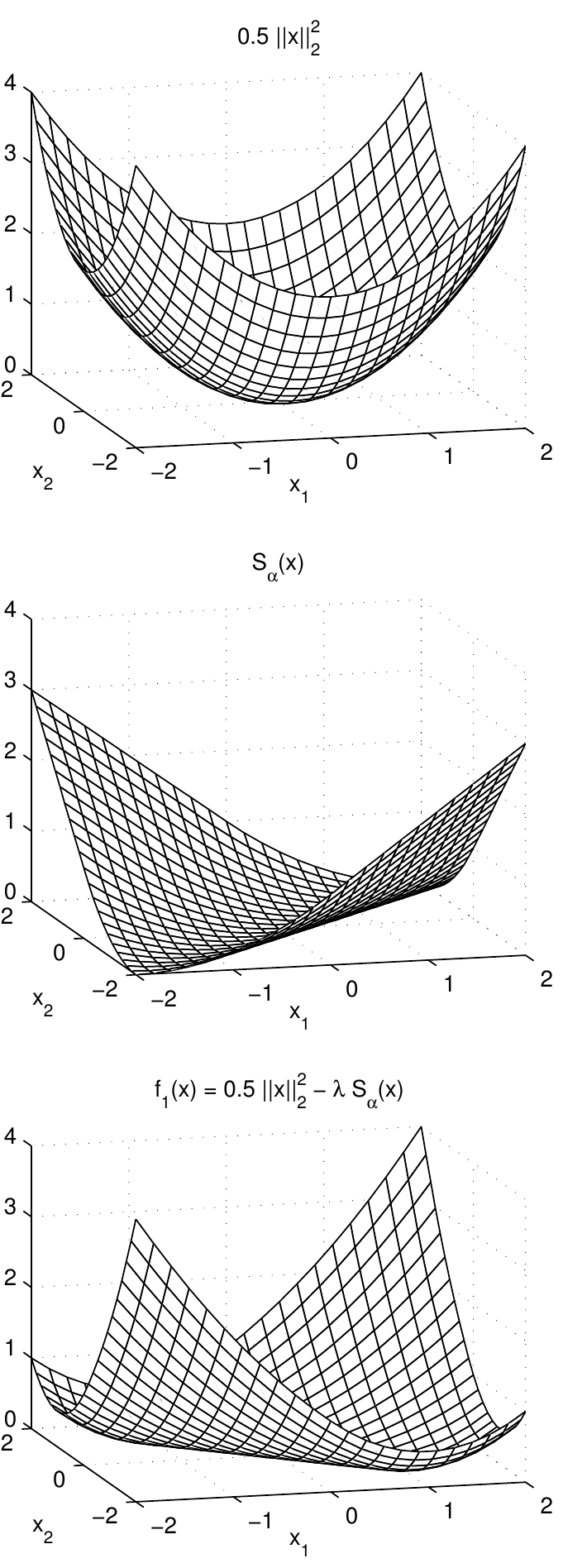}
	\caption{
		Differentiable part $ f_1 $ of cost function with $ \lam = 1 $, $ \alpha = 1 $ for $ N = 2 $.
		The function is convex as $ \alpha \le 1/\lam $.
	}
	\label{fig:pen2}
\end{figure}

\begin{figure}[t]
	\centering
	\includegraphics{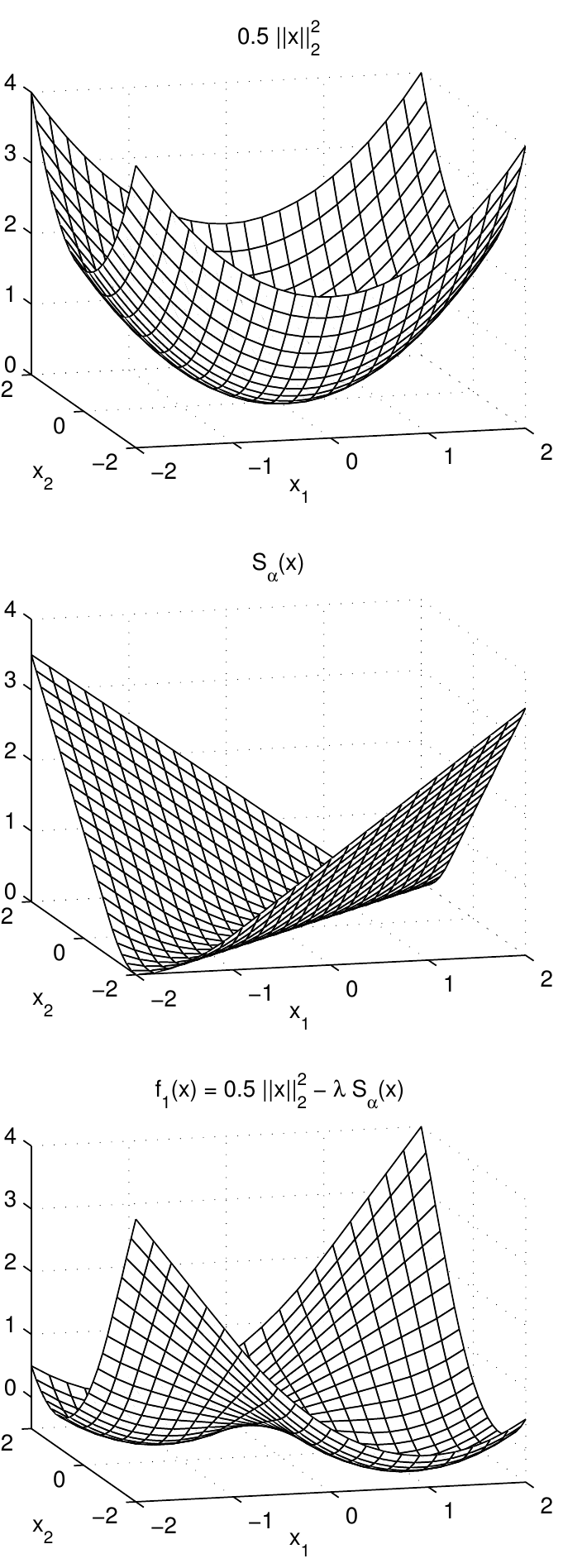}
	\caption{
		Differentiable part $ f_1 $ of cost function with $ \lam = 1 $, $ \alpha = 2 $ for $ N = 2 $.
		The function is non-convex as $ \alpha > 1/\lam $.
	}
	\label{fig:pen3}
\end{figure}

\end{document}

\clearpage

\begin{figure}
	\centering
	\includegraphics{scatter}
	\caption{
		Optimality scatter plot.
	}
	\label{fig:scatter}
\end{figure}

\end{document}